\documentclass[a4paper]{amsart}
\usepackage{amsthm}
\usepackage{amssymb}
\usepackage{amsmath}
\usepackage{verbatim,enumerate,graphicx,wrapfig,multicol,vwcol}

\usepackage{nccmath}

\renewcommand{\thefootnote}{\fnsymbol{footnote}}

\newtheorem{theorem}{Theorem}[section]
\newtheorem{lemma}[theorem]{Lemma}

\newtheorem{remark}[theorem]{Remark}

\newtheorem*{example*}{Example}

\newtheorem*{theorem*}{Theorem}

\newtheorem*{remark*}{Remark}

\newtheorem{corollary}[theorem]{Corollary}
\newtheorem*{corollary*}{Corollary}

\newtheorem{definition}[theorem]{Definition}

\newtheorem*{definition*}{Definition}

\newtheorem*{notation*}{Notation}

\numberwithin{equation}{section}

\gdef\myletter{}

\let\savetheequation\theequation
\def\theequation{\savetheequation\myletter}

\def\bt{\mathbf{t}}
\def\bp{\mathbf{p}}
\def\intt{\mathrm{int}}

\newcommand{\CC}{{\mathbb C}}

\newcommand{\RR}{{\mathbb R}}
\newcommand{\ZZ}{{\mathbb Z}}

\newcommand{\NN}{{\mathbb N}}

\newcommand{\calM}{\mathcal{M}}

\newcommand{\calN}{\mathcal{N}}

\newcommand{\calF}{\mathcal{F}}

\newcommand{\vol}{\mbox{vol}}

\def \bar{\overline}

\def \b0{{\bf 0}}

\def\calG{\mathcal{G}}

\def\calL{\mathcal{L}}

\def\Poly{\mathrm{Poly}}

\long\def\symbolfootnote[#1]#2{\begingroup%
\def\thefootnote{\fnsymbol{footnote}}\footnote[#1]{#2}\endgroup}

\begin{document}

\title{Transfinite diameter with generalized polynomial degree}

\author{Sione Ma`u}

\address{Department of Mathematics,
University of Auckland,
Auckland, NZ}
\email{s.mau@auckland.ac.nz}


\begin{abstract}
We prove a Chebyshev transform formula for a notion of (weighted) transfinite diameter that is defined using a generalized notion of polynomial degree.  We also generalize Leja points to this setting.  As an application of our main formula, we prove that in the unweighted case, these generalized Leja points recover the transfinite diameter.
\end{abstract}

\maketitle

\section{Introduction}



The usual grading of polynomials by (total) degree may be described geometrically. Let 
\begin{equation}\label{eqn:simplex}
\Sigma:=\bigl\{(x_1,\ldots,x_N)\in\RR^N_+\colon x_1,\ldots,x_N\geq 0,\ x_1+\cdots+x_N\leq 1\bigr\}
\end{equation}
 be the standard $n$-dimensional unit simplex in $\RR^N_+$.  
Then $\deg(p)\leq n$ may be reformulated as $p\in\Poly(n\Sigma)$, where
\begin{equation}\label{eqn:polynC}
\Poly(n\Sigma) \ := \   \Bigl\{ p(z) =\!\! \medop\sum_{J\in n\Sigma\cap\ZZ_+^d}  c_Jz^J\colon   c_J\in\CC  \Bigr\}.
\end{equation}

Recent results in multivariate approximation theory (\cite{boslev:bernstein}, \cite{trefethen:multivariate}) suggest that it is useful to generalize the notion of degree.  Let $C\subset\RR^N_+$ be a compact convex set which we will assume contains $\Sigma$.  Given $n\in\NN$, define $\Poly(nC)$ to be the collection of polynomials defined as in (\ref{eqn:polynC}) above with $nC$ replacing $n\Sigma$ in the sum over the exponents $J$.  The generalized degree (depending on $C$, or \emph{$C$-degree}) of a nonconstant polynomial $p$ is given by
\begin{equation}\label{eqn:cdeg}
\deg_C(p):=\min\{n\in\NN\colon p\in\Poly(nC)\},
\end{equation}
which clearly reduces to $\deg(p)$ when $C=\Sigma$.

The classical theory of polynomial approximation and interpolation in $\CC^N$ involves fundamental notions of pluripotential theory: the \emph{Siciak-Zaharjuta extremal function} associated to a compact set $K\subset\CC^N$ 
\begin{equation} \label{eqn:VKdef}
V_K(z):= \sup\left\{ \frac{1}{\deg p}\log^+|p(z)|\colon \|p\|_K\leq 1\right\},
\end{equation}
\emph{complex equilibrium measure $(dd^cV_K)^N$}, and \emph{Fekete-Leja transfinite diameter}.   All of these quantities are fundamentally related to the classical notion of degree.  They must be appropriately generalized if one wants to prove approximation theorems involving $\deg_C$ or $\Poly(nC)$.  

This motivates recent work on the so-called $C$-pluripotential theory.  The spaces $\Poly(nC)$ and  \emph{$C$-extremal function} 
$$
V_{C,K}(z):= \sup\left\{ \frac{1}{\deg_C p}\log^+|p(z)|\colon \|p\|_K\leq 1\right\}
$$
were defined in \cite{bayraktar:zero} and further studied in \cite{boslev:bernstein}.\footnote{In those papers the convex body is denoted $P$ and the terminology $P$-pluripotential theory, $P$-extremal function, etc. is used.}   The  \emph{equilibrium measure $(dd^cV_{C,K})^N$}  and \emph{$C$-transfinite diameter} $\delta_C(K)$ were studied in \cite{bayraktarbloomlev:pluripotential}, using the $L^2$ energy methods of Berman-Boucksom pioneered in \cite{bermanboucksom:growth},\cite{bermanboucksomnystrom:fekete}.

The Berman-Boucksom methods work more naturally in a general weighted setting involving an \emph{admissible weight function $w\colon K\to[0,\infty)$} and its logarithm 
$Q=-\log|w|$.  These give weighted versions of all of the above notions, denoted  $V_{C,K,Q}(z)$, $(dd^cV_{C,K,Q})^N$, and $\delta^w_C(K)$.  Setting $w\equiv 1$ recovers the unweighted case.

In this paper, we study the weighted $C$-transfinite diameter $\delta^w_C(K)$.  Our main theorem is the following.
\begin{theorem*}[Theorem \ref{thm:Ok}] 
Let $K\subset\CC^N$ be compact and $w\colon K\to[0,\infty)$ an admissible weight.  Then we have the formula 
$$
\delta^w_C(K)^{A_N} \ = \  \exp\left(
\frac{1}{\vol_N(C)}\int_{\intt(C)} \log T^{w}_{\prec}(K,\theta)\, d\theta  \right).
$$
\end{theorem*} 
\noindent Here, $\vol_N$ denotes real $N$-dimensional volume, $d\theta$  the corresponding volume measure where $\theta=(\theta_1,\ldots,\theta_N)$ are coordinates in $\RR^N$,  and $A_N$ is a normalization constant.  The quantity $T^w_{\prec}(K,\theta)$ is called a \emph{directional Chebyshev constant (for the direction $\theta$)}, and the function $$\intt(C)\ni\theta\mapsto \log T^w_{\prec}(K,\theta)\in[0,\infty)$$ is called the \emph{Chebyshev transform}. (We follow the terminology in \cite{bloomlev:weighted} and \cite{nystrom:transforming}.)

In the next section (Section 2), we construct directional Chebyshev constants and the Chebyshev transform by a limiting process using \emph{submultiplicative functions} (see Definition \ref{def:submult}).  Our overall method is standard, but we encounter a new technicality in the $C$-theory.  The grevlex ordering on monomials,  denoted in this paper by $\prec$  (see Definition \ref{def:grevlex}), satisfies:
\begin{enumerate}
\item If $\deg(z^{\alpha})<\deg(z^{\beta})$ then $z^{\alpha}\prec z^{\beta}$.  
\item If $z^{\alpha}\prec z^{\beta}$ then $z^{\alpha+\gamma}\prec z^{\beta+\gamma}$ for each $\gamma$.
\end{enumerate}
Property (1) is needed for theorems that involve the notion of degree, and property (2) is needed to create submultiplicative functions.  If we replace $\deg$ with $\deg_C$, then statement (1) fails for the grevlex ordering $\prec$;  consider replacing it with some global ordering on monomials (denoted by, say, $\prec_C$) such that the following analog holds:
\begin{itemize}
\item[$(1_C)$] If $\deg_C(z^{\alpha})\prec_C\deg_C(z^{\beta})$ then $z^{\alpha}\prec_C z^{\beta}$.
\end{itemize}
It is easy to verify that in cases where  $C$ is \emph{not} a simplex,  one can find $\alpha,\beta,\gamma$ such that 
 $z^{\alpha}\prec_C z^{\beta}$ but $z^{\alpha+\gamma}\not\prec_C z^{\beta+\gamma}$, i.e., the corresponding statement $(2_C)$  fails.
However, to define directional Chebyshev constants via a limiting process one can get away with a weaker property.\footnote{But note that submultiplicativity is needed to deduce other properties of these constants, cf. Lemmas \ref{lem:06} and \ref{lem:riemann}.}

  In Section 2 we use the grevlex ordering $\prec$ to construct a submultiplicative function using discrete Chebyshev constants $T_k^{\prec}(\nu_k,\alpha)$.  We construct directional Chebyshev constants $T_{\prec}^w(K,\theta)$ by a limiting process as $k\to\infty$.  Then in Section 3, we construct  discrete Chebyshev constants ${T^C_k}(\nu_k,\alpha)$ using a global ordering $\prec_C$ satisfying $(1_C)$, as well as the associated directional Chebyshev constants $T^w_{C}(K,\theta)$.   The ordering $\prec_C$ needs to satisfy a certain technical property (property ($\dag$), that holds for a generic convex body $C$) in order to prove that  $T^w_{C}(K,\theta)$ is given by a well-defined limit.\footnote{We conjecture that the limit exists in general.}

 In Section 4, Theorem \ref{thm:Ok} is proved by suitably adapting the arguments in \cite{mau:newton}. Then, the same estimates and some measure-theoretic arguments yield a similar theorem with $\log T^w_C(K,\theta)$ replacing $\log T^w_{\prec}(K,\theta)$ in the integral formula.
 
 
Finally, in Section 5, we generalize the construction of Leja points, using the ordering $\prec_{C}$.   In the unweighted case ($w\equiv 1$)  we can use our estimates in Section 4 to deduce the asymptotic behaviour of these so-called \emph{$C$-Leja points}.  We show that in the limit, they give the $C$-transfinite diameter.

\section{Chebyshev constants from a submultiplicative function}

In this section we construct a submultiplicative function from sup norms of monic polynomials.  We then construct directional Chebyshev constants using this function.  The basic properties of this construction are well-known, going  back to Zaharjuta \cite{Zaharjuta:transfinite}.  They are given in  Lemma \ref{lem:06}, whose  proof is omitted.\footnote{But see the paragraph above Lemmas 1.1--1.2 in \cite{bloomlev:weighted}.} 

   We assume throughout that our convex body $C\subset\RR_+^N$ 
 contains the simplex 
$\Sigma$ defined as in  (\ref{eqn:simplex}).  
By a calculation,  
\begin{equation}\label{eqn:+} r_1C+ r_2C\subseteq (r_1+r_2)C \end{equation}
holds for $r_1,r_2\in\RR_+$.

When $k\in\ZZ_+$ we obtain a grading of the polynomials $\CC[z]=\bigcup_k\Poly(kC)$, and   
\begin{equation}\label{eqn:grade}
\Poly(k_1C)\Poly(k_2C) \subseteq \Poly((k_1+k_2)C),\quad k_1,k_2\in\ZZ_+
\end{equation}
follows from (\ref{eqn:+}).

\begin{definition}\label{def:grevlex}  \rm Let $\prec$ denote the \emph{grevlex ordering on $\ZZ^N$}, in which we set $\alpha\prec\beta$ if 
\begin{itemize}
\item $|\alpha|<|\beta|$; or
\item $|\alpha|=|\beta|$, and there exists $k\in\{1,\ldots,N\}$ such that $\alpha_k<\beta_k$ and  $\alpha_j=\beta_j$ for all $j<k$.
\end{itemize} 
For monomials, we set $z^{\alpha}\prec z^{\beta}$ if $\alpha\prec\beta$. 
\end{definition}

It is straightforward to verify that $\prec$ is compatible with addition. 

\begin{lemma}\label{lem:precadd}
If $\alpha\prec\beta$ then
$\alpha+\gamma \prec \beta + \gamma \prec \beta +\delta$  for all $\gamma,\delta\in\NN^N$ with $\gamma\prec\delta$.  \qed
\end{lemma}

We next construct classes of monic polynomials.  
  Given $k\in\NN$ and $\alpha\in (kC\cap\ZZ^N)$, let  
\begin{equation}\label{eqn:Mgrevlex}
\calM^{\prec}_k(\alpha):= \Bigl\{ p\in\Poly(kC)   \colon p(z) = z^{\alpha} + 
\medop\sum_{\substack{\beta\in kC\cap\ZZ^N \\ \beta\prec\alpha}}  a_{\beta}z^{\beta}  \Bigr\}.
\end{equation}

\begin{lemma}\label{lem:monic}
Let $C\subset\RR_+^N$ be a convex body containing the standard simplex $\Sigma$, and let $k_1,k_2\in\NN$.  Then  
$$
\calM^{\prec}_{k_1}(\alpha_1)\calM^{\prec}_{k_2}(\alpha_2) \subseteq \calM^{\prec}_{k_1+k_2}(\alpha_1+\alpha_2)
$$ 
for all $\alpha_1\in k_1C$ and $\alpha_2\in k_2C$.
\end{lemma}

\begin{proof}
Let $\displaystyle p=z^{\alpha_1} + \sum_{\substack{\beta\in k_1C\\ \beta\prec\alpha_1}} a_{\beta}z^{\beta} \hbox{ and }q=z^{\alpha_2} + \sum_{\substack{\gamma\in k_2C\\ \gamma\prec\alpha_2}} c_{\gamma}z^{\gamma}.$   Write  
$$
pq= z^{\alpha_1+\alpha_2} + \sum_{\beta,\gamma}a_{\beta}c_{\gamma}z^{\beta+\gamma}.
$$
By (\ref{eqn:+}), $\beta+\gamma\in(k_1+k_2)C\cap\ZZ^N$ for all $\beta,\gamma$; and by Lemma \ref{lem:precadd},  $\alpha_1+\alpha_2\prec \beta+\gamma$ for all $\beta,\gamma$ in the sum.   Hence $pq\in\calM^{\prec}_{k_1+k_2}(\alpha_1+\alpha_2)$.
\end{proof}

Recall that $p\in\Poly(kC)$ may be identified with the (entire) function $a\mapsto p(a)$ given by the usual evaluation of polynomials.  But we can also fix a function $\varphi\colon\CC^N\to\CC$ (a \emph{weight}) and identify $p\in\Poly(kC)$ with $a\mapsto \varphi(a)^kp(a)$ (a \emph{weighted} evaluation), and use this to define function-theoretic norms.    

\medskip

In what follows, $\nu_k$ denotes a norm on $\Poly(kC)$.  Examples of norms are:
\begin{enumerate}
\item $\nu_k(p)= \|\varphi^k p\|_K$, where $K\subset\CC^n$ is compact and $\varphi\colon K\to\CC$ is  continuous. 
\item $\nu_k(p) = \left(\medop\int_{\!\! K} |\varphi^k p|^2d\mu\right)^{1/2}$, where $\varphi,K$ are as in (1) and $\mu$ is a finite measure on $K$.
\end{enumerate}
Suppose for each $k\in\NN$, $\nu_k$ is as in either (1) or (2).  Then we have the following.

\begin{lemma}
 Let $k,m\in\NN$, let $p\in\Poly(kC)$ and let $q\in\Poly(mC)$.  Then $pq\in\Poly((k+m)C)$ and 
\begin{equation}\label{eqn:nu}
\nu_{k+m}(pq)\ \leq  \nu_k(p) \nu_m(q).
\end{equation}
\end{lemma}

\begin{proof}
If $\nu_k,\nu_m$ are as in (1), then  
$
\|\varphi^{k+m}pq\|_K \leq \|\varphi^kp\|_K\|\varphi^mq\|_K$ and the conclusion follows.  
If the norms are as in (2), use the Cauchy-Schwarz inequality.
\end{proof}

For each $k\in\NN$ let $\nu_k$ be a norm on $\Poly(kC)$.  The sequence of norms $\{\nu_k\}_{k=1}^{\infty}$ is  \emph{submultiplicative} if (\ref{eqn:nu}) holds for every $k,m\in\NN$. 

\begin{definition}\rm 
 Define the \emph{discrete Chebyshev constant $T^{\prec}_k(\nu_k,\alpha)$} by 
$$
T^{\prec}_k(\nu_k,\alpha):= \inf\bigl\{ \nu_k(p)\colon   p\in\calM^{\prec}_k(\alpha)  \bigr\}^{1/k}.
$$
A polynomial $t$ for which $\deg_C(t)=k$ and $T^{\prec}_k(\nu_k,\alpha) = \nu_k(t)^{1/k}$ will be called a \emph{Chebyshev polynomial of degree $k$.}
\end{definition}

\begin{lemma} \label{lem:subm}
If $\{\nu_k\}_{k=1}^{\infty}$ is a submultiplicative sequence of norms, then 
$$
T^{\prec}_{k_1+k_2}(\nu_{k_1+k_2},\alpha_1+\alpha_2)^{k_1+k_2}   \leq    T^{\prec}_{k_1}(\nu_{k_1},\alpha_1)^{k_1}T^{\prec}_{k_2}(\nu_{k_2},\alpha_2)^{k_2}  .
$$
\end{lemma}

\begin{proof}
Let $t_{k_1},t_{k_2}$ be polynomials satisfying
\begin{equation}\label{eqn:l4}\nu_{k_1}(t_{k_1})=T^{\prec}_{k_1}(\nu_{k_1},\alpha_1)^{k_1}  \hbox{ and } \nu_{k_2}(t_{k_2})=T^{\prec}_{k_2}(\nu_{k_2},\alpha_2)^{k_2}.\end{equation} Then $t_{k_1}t_{k_2}\in\calM^{\prec}(k_1+k_2,\alpha_1+\alpha_2)$ by Lemma \ref{lem:precadd}, and 
$$
T^{\prec}_{k_1+k_2}(\nu_{k_1+k_2},\alpha_1+\alpha_2)^{k_1+k_2} \leq  \nu_{k_1+k_2}(t_{k_1}t_{k_2}) \leq \nu_{k_1}(t_1)\nu_{k_2}(t_{k_2}).
$$
by submultiplicativity.  Now put (\ref{eqn:l4}) into the above.
\end{proof}

Let $\Omega\subset\ZZ_+\times\ZZ^N_+$ be a set that is closed under addition (i.e., $(k,\alpha),(l,\beta)\in\Omega$ implies $(k+l,\alpha+\beta)\in\Omega$).    

\begin{definition}\rm  \label{def:submult}
A function $Y:\Omega\to[0,\infty)$ is called \emph{submultiplicative} if
\begin{equation}\label{eqn:submult}
Y(k+l,\alpha+\beta)\leq Y(k,\alpha)Y(l,\beta).
\end{equation}
\end{definition}

As an immediative consequence of Lemma \ref{lem:subm}, we have the following.
\begin{corollary}
The function $(k,\alpha)\mapsto T_{k}(\nu_k,\alpha)^k$ is submultiplicative on the set
$$
\Omega_C := \{(k,\alpha)\colon k\in\ZZ_+,\alpha\in kC\cap\ZZ^N   \}.
$$
\qed
\end{corollary}

A submultiplicative function on $\Omega_C$ yields a real convex function on $\intt(C)$ by a limiting process.  \def\Yinf{Y_{\infty}}
\begin{lemma}\label{lem:06}
Let $Y\colon\Omega_C\to[0,\infty)$ be submultiplicative. 
\begin{enumerate}
\item The limit 
$$ Y_{\infty}(\theta):=   \lim_{\substack{k\to\infty\\ \alpha/k\to\theta}}  Y(k,\alpha)^{1/k}$$
exists for all $\theta\in \intt(C)$.
\item We have 
\begin{equation}\label{eqn:Yconv}
\Yinf(t\theta+(1-t)\phi) \leq  \Yinf(\theta)^t\Yinf(\phi)^{1-t}  \hbox{ for all $\theta,\phi\in\intt(C)$ and $t\in[0,1]$.}  \end{equation}  
Hence $\intt(C)\ni\theta\mapsto\log\Yinf(\theta)\in\RR$ is a convex function.  \qed 
\end{enumerate}
\end{lemma}


In particular, part (2) shows that $\theta\mapsto\log\Yinf(\theta)$ is continuous, 
and hence uniformly continuous on compact subsets of $\RR^N_+$.  An exercise in the triangle inequality then yields the following. 
\begin{lemma}\label{lem:riemann}
Let $Q\subset\intt(C)$ be compact. For each $k\in\NN$, consider the nodes $Q\cap\frac{1}{k}\ZZ^N$ under some enumeration $\{\theta_j\}_{j=1}^{N_k}$, with $\frac{1}{k}\alpha_{j,k}=\theta_j$.  Then 
$$
\sup\left\{ |\log Y(\alpha_{j,k})^{1/k}-\log\Yinf(\theta_j)|\colon j\in\{1,\ldots,N_k\} \right\}\longrightarrow 0 \hbox{ as } k\to\infty. 
$$ \qed
\end{lemma}

Fix a compact set $K\subset\CC^N$ and a continuous weight $w\colon K\to[0,\infty)$.  With the submultiplicative sequence of norms $\nu_k(p):=\|w^kp\|_K$, construct the submultiplicative function $Y(k,\alpha):=T^{\prec}_k(\nu_k,\alpha)^k$ on $\Omega_C$.    Then for each $\theta\in \intt(C)$,  set $T^{w}_{{\prec}}(K,\theta):=\Yinf(\theta)$, as constructed in Lemma \ref{lem:06}.  This  yields the following corollary.

\begin{corollary}\label{cor:dircheby}
Lemmas \ref{lem:06} and \ref{lem:riemann} hold with $T_k^{\prec}(\nu_k,\alpha)^k, T_{\prec}^w(K,\theta)$ replacing $Y(k,\alpha),Y_{\infty}(\theta)$. \qed
\end{corollary}


\begin{definition}\rm
The number $T^w_{\prec}(K,\theta)$ is called a \emph{directional Chebyshev constant (with direction $\theta$) associated to $(K,w)$}, and the function $\theta\mapsto\log T^{w}_{{\prec}}(K,\theta)$ is called the \emph{Chebyshev transform}. 
\end{definition}

\section{Chebyshev constants compatible with generalized degree} \label{sec:3}

Now consider a \emph{modified grevlex ordering} that is compatible with the grading given by the convex body $C$.  
First, for $\alpha\in\ZZ_+^N$ define $$r(\alpha):=\inf\{r\in\RR_+\colon \alpha\in rC \}.$$ 
Then for $j\in\ZZ_+$, $r(j\alpha)=j r(\alpha)$ is immediate, and 
$$
r(\alpha+\beta) \leq r(\alpha)+ r(\beta), \quad \alpha,\beta\in\ZZ_+^N
$$
follows from (\ref{eqn:+}).  Thus 
\begin{equation}\label{eqn:jk}
r(j\alpha +k\beta) \leq jr(\alpha) + kr(\beta),\quad \alpha,\beta\in\ZZ_+^N \hbox{ and } j,k\in\ZZ_+.
\end{equation}

\begin{definition}\rm \label{def:mgrevlex}
Define $\prec_C$ on $\ZZ_+^N$ by setting $\alpha\prec_C\beta$ (and equivalently, $z^{\alpha}\prec_C z^{\beta}$) if 
\begin{itemize}
\item $r(\alpha) < r(\beta)$; or
\item $r(\alpha)=r(\beta)$ and $z^{\alpha}\prec z^{\beta}$ (where as before, $\prec$ denotes grevlex).
\end{itemize}\end{definition}
This is compatible with the grading: if $\alpha\prec_C\beta$ then $r(\alpha)\leq r(\beta)$,  so $$\deg_C(z^{\alpha})=\lceil r(\alpha) \rceil \leq \lceil r(\beta)\rceil = \deg_C(z^{\beta}).$$  

For $\alpha\in kC$, define the class of monic polynomials
\begin{equation*}
\calM^{C}_k(\alpha) := \Bigl\{ p\in\Poly(kC)\colon p(z) = z^{\alpha} + 
\medop\sum_{\substack{\beta\in kC\cap\ZZ^N \\ \beta\prec_C\alpha}}  a_{\beta}z^{\beta}  \Bigr\} 
\end{equation*}
and the associated (discrete) Chebyshev constant
\begin{align*} 
T_k^C(\nu_k,\alpha) &:= \inf\{ \nu_k(p)\colon p\in\calM^C_k(\alpha)\}^{1/k},
\end{align*}
where $\nu_k(p) = \|w^kp\|_K$, as before.  

As mentioned in the Introduction: $T^C_k(\nu_k,\alpha)$ does not yield a submultiplicative function (unlike $T_k^{\prec}(\nu_k,\alpha)$), essentially because $\calM_k^C(\alpha)$ uses the ordering $\prec_C$ which is not compatible with addition of exponents.   Hence we cannot apply Lemmas \ref{lem:06} and \ref{lem:riemann} directly.  However, we do have the following.  
\begin{lemma}
Let $\alpha\in kC$.  Then $p\in\calM^C_k(\alpha)$ implies $p^j\in\calM^C_{jk}(j\alpha)$ for $j\in\mathbb{Z}_+$.
\end{lemma}

\begin{proof}
Write
\begin{eqnarray*}
p(z)  &=&  z^{\alpha} + \sum_{\substack{\beta\in kC\cap\ZZ^N\\ \beta\prec_C \alpha }} a_{\beta}z^{\beta}.
\end{eqnarray*}

We claim that the leading term of $p^j$ with respect to $\prec_C$ is  $z^{j\alpha}$.  The monomial of any other term is of the form $z^{k\alpha+(j-k)\beta}$ for some $\beta\prec_C\alpha$ and   $k<j$.  If $r(\beta)<r(\alpha)$ then
$$
r(k\alpha+(j-k)\beta)\leq kr(\alpha) +(j-k)r(\beta) < jr(\alpha) = r(j\alpha), 
$$
So $k\alpha+(j-k)\beta\prec_C j\alpha$. 

If $r(\beta)=r(\alpha)$ then similarly,   
$r(k\alpha+(j-k)\beta)\leq r(j\alpha).$  We also have $\beta\prec\alpha$.  The latter gives  
$k\alpha+(j-k)\beta\prec k\alpha$ since grevlex is compatible with addition.  Altogether, we can conclude that $k\alpha+(j-k)\beta\prec_Cj\alpha$ holds in all cases.  Thus the leading term of $p^j$ is $z^{j\alpha}$, so $p^j\in\calM_{jk}^C(j\alpha)$.  
\end{proof}

We need a slightly stronger property than  $p\in\calM^C_k(\alpha) \Longrightarrow p^j\in\calM^C_{jk}(j\alpha)$  to apply the standard argument showing that the constants $T^C_k(\nu_k,\alpha)$ give directional limits: 

\begin{itemize}
\item[($\dag$)] {\it Given $k,l\in\ZZ_+$, $\alpha\in kC$, and $p\in\calM^C_k(\alpha)$,  there exists $j_0\in\ZZ_+$ such that for any $\delta\in lC$,  }
$$
z^{\delta}p^j\in\calM^C_{jk+l}(j\alpha+\delta) \ \hbox{ for all } j\geq j_0.
$$
\end{itemize}

If $C$ is a simplex, then $\prec_C$ is actually submultiplicative, which easily implies ($\dag$).  We can see this by elementary geometry. Put $r_1=r(\alpha)$ and $r_2=r(\alpha +\delta)$.  Then $\partial(r_1C)\cap\RR_+^N$ and $\partial(r_2C)\cap\RR_+^N$ are on parallel hyperplanes, so translation by $\delta$ sends all points of the first hyperplane into the second.  It follows that $r(\beta)\leq r(\alpha)$ implies $r(\beta+\delta)\leq r(\alpha+\delta)$, and hence, $\beta\prec_C\alpha$ implies $\beta+\delta\prec_C\alpha+\delta$.

\medskip

Another case for which ($\dag$) holds is when the scaled copies $rC$ ($r\in\RR_+$) cover new points of $\ZZ_+^N$ one at a time as $r\to\infty$.  In other words, whenever $\alpha\neq\beta$ we have  $r(\alpha)\neq r(\beta)$.  (Clearly, this condition holds for a generic convex body $C$.)

To verify ($\dag$) in this case, let $N(p)$ be the \emph{Newton polytope} of $p\in\calM^C_k(\alpha)$, i.e., the convex hull in $\RR_+^N$ of all $\beta\in\ZZ_{\geq 0}^N$ such that $p$ contains the monomial $z^{\beta}$.    Put $r=r(\alpha)$; then by hypothesis, $r(\beta)<r$ for all other terms of $p$ (since $z^{\alpha}$ is the leading term), so $\alpha$ is the only point of $N(p)$ on $\partial(rC)\cap\RR_+^N$.  

For $s>0$, a translation of $N(p)$ by $s\delta$ takes $\alpha$ to a point $\alpha+s\delta$ on $\partial(\tilde rC)$ for some $\tilde r>r$.   If we take $s$ sufficiently close to zero, then only points in a small neighborhood of $\alpha$ will move into the complement of $rC$.  We can consider a neighborhood small enough so that $\partial(rC)$ and $\partial(\tilde rC)$ are approximated by parallel hyperplanes.  Then, a translation of $N(p)$ by $s\delta$ sends $\alpha$ to $\partial(\tilde rC)$, but all other points of $N(p)$ remain in the interior of $\tilde rC$.  

Now consider taking $s=\tfrac{1}{j}$ for $j\in\ZZ_+$ sufficiently large.  Next, note that $N(p^j)=jN(p)$ by a straightforward calculation.  With the polynomial $p^j$, everything in the previous paragraph scales by a factor of $j$, i.e.,  $N (p^j)=jN(p)\subset jrC$, and $j\alpha+ \delta$ is the only point of $N(p^j)$ that translates to $\partial(j\tilde rC)$; all other points remain in the interior of $j\tilde rC$.  See the picture below.

\bigskip

\includegraphics[height=7cm]{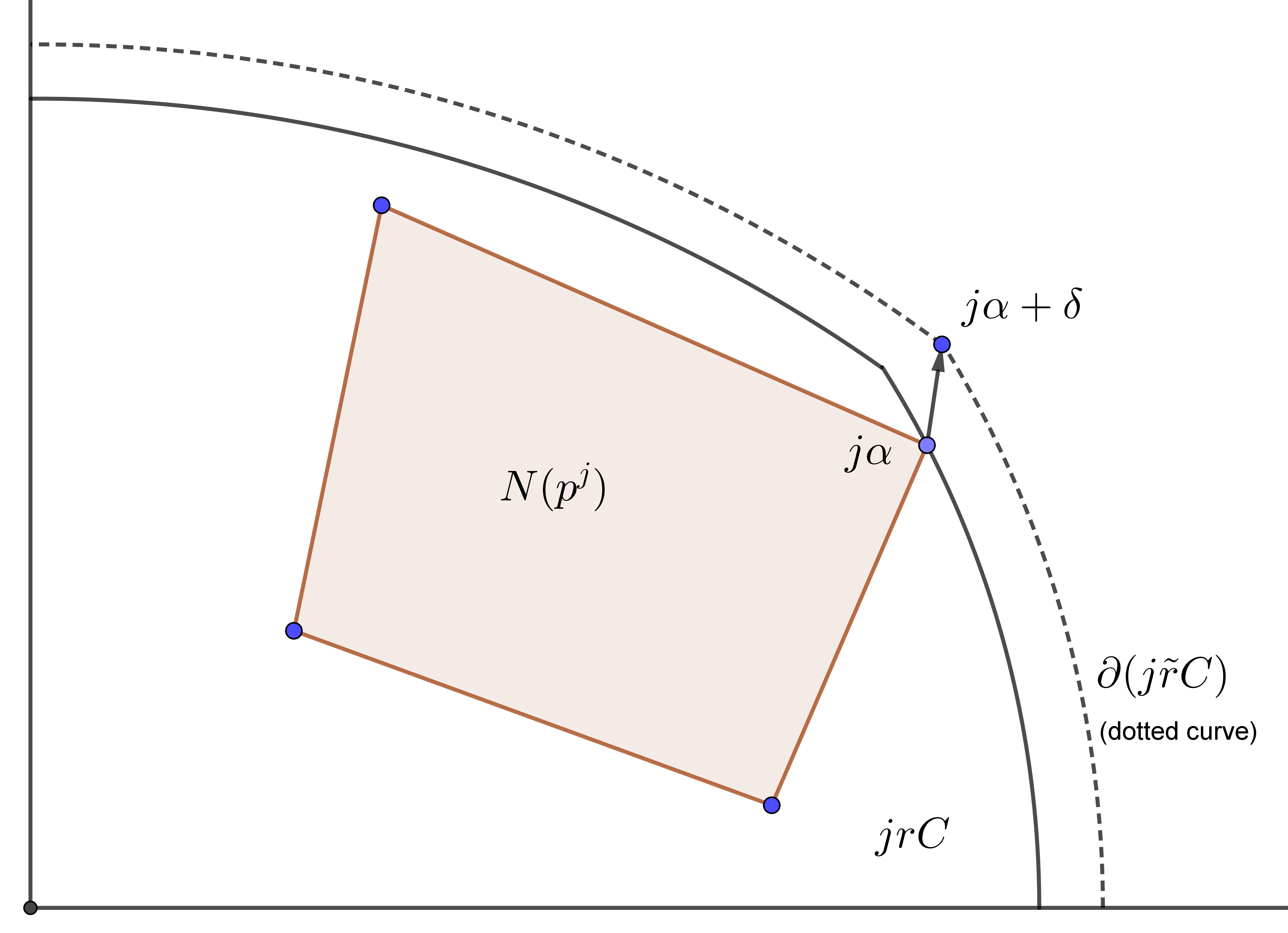}

\bigskip

 We have $N(p^j)+\delta = N(z^{\delta}p^j)$, so the leading term of $z^{\delta}p^j$ is given by $z^{\delta+j\alpha}$.  
Thus if $z^{\alpha}$ is the leading monomial of $p$, then $z^{j\alpha+\delta}$ is the leading monomial of $z^{\delta}p^j$, i.e., $z^{\delta}p^j\in\calM^C_{jk+l}(j\alpha+\delta)$.  This verifies condition ($\dag$).  

\medskip

Applying the argument of Lemma 1 in \cite{Zaharjuta:transfinite}, we have the following.

\begin{lemma}\label{lem:32}
Suppose condition ($\dag$) holds for $C$.  Then the limit
$$
T_C^{w}(K,\theta):= \lim_{\substack{k\to\infty\\ \alpha/k\to\theta }} T_k^C(\nu_k,\alpha)
$$
exists for all $\theta\in\intt(C)$.  \qed
\end{lemma}

We leave it as an open problem to relate $T^w_{\prec}(K,\theta)$ to $T^w_{C}(K,\theta)$.  \emph{Are these directional Chebyshev constants equal?}


\section{Vandermonde determinant estimates}



In this section, the compact set $K$ and continuous weight $w\colon K\to[0,\infty)$ are fixed.    We assume condition $(\dag)$ of the previous section holds for $C$.  We also use the following notation:
\begin{align*}
M_k &= \hbox{ number of points in $kC\cap\ZZ^N$} = \dim\Poly(kC), \\
h_k &=\hbox{ number of points in $(kC\setminus(k-1)C)\cap\ZZ^N$}, \\
L_k &=\medop\sum_{\alpha\in kC } \deg(z^{\alpha}).
\end{align*}

\begin{remark} \label{rem:lk} \rm
Note that $L_k$ is the sum of the ordinary degrees of the monomials, not the $C$-degrees.  The normalization in equation (\ref{eqn:dk}) below uses an $L_k$-th root so that transfinite diameter scales like a length.
\end{remark}

We have the following limits.  The first two are easy, and the last one is a bit more involved---an explicit calculation is given in \cite{bayraktarbloomlev:pluripotential}  (see Remark 2.17 of that paper.

\begin{lemma}\label{lem:limits}
 As $k\to\infty$,  
$$
\frac{h_k}{M_k}\to 0,\quad   (M_k!)^{1/kM_k}\to 1,\quad  \hbox{and }\frac{L_k}{kM_k}\to A_N, 
$$
where 
$$A_N = \frac{1}{\vol_N(C)}\int_{C} (\theta_1+\cdots+\theta_N) d\theta.$$
Here $d\theta=d\theta_1\wedge\cdots\wedge d\theta_N$ denotes the $N$-dimensional volume measure in $\RR^N$ with coordinates $\theta=(\theta_1,\ldots,\theta_N)$.
\qed
\end{lemma}

\begin{definition} \label{def:vand}  \rm
Let $k\in\NN$, and let $\{\alpha(j)\}_{j=1}^{M_k}$ be an enumeration of $kC\cap\ZZ^N$, so that $\{z^{\alpha(j)}\}_{j=1}^{M_k}$ is a basis of $\Poly(kC)$.    Given a finite collection of points $\{\zeta_1,\ldots,\zeta_s\}$ with $s\leq M_k$, define the \emph{Vandermonde matrix}  
\begin{equation*}
VM_{k}^w(\zeta_1,\ldots,\zeta_s)  :=      \begin{bmatrix} w(\zeta_1)^kz^{\alpha_1}(\zeta_1) &   \cdots &  w(\zeta_s)^kz^{\alpha_1}(\zeta_s) \\
\vdots & \ddots & \vdots \\
 w(\zeta_1)^kz^{\alpha_s}(\zeta_1) & \cdots &  w(\zeta_s)^kz^{\alpha_s}(\zeta_s) \end{bmatrix}   =    [ w(\zeta_l)^kz^{\alpha_j}(\zeta_l)]_{j,l=1}^s  , 
\end{equation*}
as well as the   \emph{Vandermonde determinant} 
\begin{equation}\label{eqn:VDM}
VDM^w_{k}(\zeta_1,\ldots,\zeta_s):= \det \bigl(VM_{k}^w(\zeta_1,\ldots,\zeta_s)\bigr). \end{equation}
For a set $K\subset V$, define
\begin{equation}\label{eqn:Vk}
V_{k}^w(K,s) := \sup\{|VDM_{k}^w(\zeta_1,\ldots,\zeta_s)|: \{\zeta_1,\ldots,\zeta_s\}\subset K \}.
\end{equation}
The quantity  
\begin{equation}\label{eqn:dk} \delta_{C,k}^w(K):= \left(V_{k}^w(K,M_k)\right)^{1/L_k} \hbox { is the \emph{$k$-th order $C$-diameter of $K$}},
\end{equation}  and 
$$\delta^w_C(K):=\limsup_{k\to\infty} \delta_{C,k}^w(K)  \hbox{ is the \emph{(weighted) $C$-transfinite diameter of $K$}}.$$  
\end{definition}

We compute upper and lower bounds for $V_k^w(K,M_k)$.

\begin{lemma}\label{lem:<}
Fix $k\in\NN$.  The inequality
\begin{equation*}
V_{k}^w(K,M_k)  \leq M_k!\prod_{j=1}^{M_k} T_k^{\prec}(\nu_k,\alpha(j))^k
\end{equation*}
holds.  
The same inequality holds  if we replace each factor $T_k^{\prec}(\nu_k,\alpha(j))$ on the right-hand side by $T_k^C(\nu_k,\alpha(j))$.
\end{lemma}

\begin{proof} 
Enumerate the monomials $\{z^{\alpha(j)}\}_{j=1}^{M_k}$ according to the grevlex ordering, and for each $j\in\{1,\ldots,M_k\}$, let
$$\bt_{j}(z) = z^{\alpha(j)} + \medop\sum\nolimits_{l\prec j}c_lz^{\alpha(l)}  $$ denote the Chebyshev
polynomial in $\calM_k(\alpha(j))$, i.e., $\|w^k\bt_{j}\|_K = T_k^{\prec}(K,\alpha(j))^k$.  Now choose a set of $M_k$ points such that $V_{k}^w(K,M_k) = |VDM_{k}^w(\zeta_1,\ldots,\zeta_{M_k})|$; then 
$$
VDM_{k}^w(\zeta_1,\ldots,\zeta_{M_k})   
=   \det \begin{bmatrix}
w^k(\zeta_1)z^{\alpha(1)}(\zeta_1) & \cdots & w^k(\zeta_{M_k})z^{\alpha(1)}(\zeta_{M_k})  \\
\vdots & \ddots & \vdots \\
w^k(\zeta_1)z^{\alpha(M_k-1)}(\zeta_1) & \cdots & w^k(\zeta_{M_k})z^{\alpha(M_k-1)}(\zeta_{M_k})  \\
w^k(\zeta_1)\bt_{M_k}(\zeta_1) & \cdots & w^k(\zeta_1)\bt_{M_k}(\zeta_{M_k})  
\end{bmatrix}   
$$
since we may add multiples of other rows to the last, leaving the determinant unchanged.  Expanding  and using the triangle inequality, we obtain 
\begin{align*} 
V_{k}^w(K,M_k) &\leq \sum_{j=1}^{M_k} |w^k(\zeta_{j})\bt_{M_k}(\zeta_{j})|\cdot |VDM_k^w(\zeta_1,\ldots,\zeta_{j-1},\zeta_{j+1},\ldots,\zeta_{M_k})|     \\
& \leq M_k\cdot  T^{\prec}_{k}(\nu_k,\alpha(M_k))^kV_{k}^w(K,M_{k}-1).
\end{align*}
A similar argument may be used to obtain
$$
V_{k}^w(K,M_{k}-1)   \leq (M_k-1)\cdot  T^{\prec}_{k}(\nu_k,\alpha(M_k-1))^kV_{k}^w(K,M_{k}-2),
$$
and further,  $V_{k}^w(K,j) \leq j\cdot  T^{\prec}_{k}(\nu_k,\alpha(j))^k  V_{k}^w(K,j-1)$ for any smaller $j$.
Putting all of these inequalities together gives the desired inequality. 

 For the inequality involving the $T_k^C(\nu_k,\alpha(j))^k$ constants, repeat the proof using $\prec_C$ to enumerate the monomials.   
\end{proof}

\begin{lemma}\label{lem:>}
Fix $k\in\NN$.  Then
$$
V_{k}^w(K,M_k) \geq \prod_{j=1}^{M_k} T_k^{\prec}(\nu_k,\alpha(j))^{k}.
$$
The same result holds if we replace each factor $T_k^{\prec}(\nu_k,\alpha(j))$ on the right-hand side by $T_k^C(\nu_k,\alpha(j))$.
\end{lemma}

\begin{proof}
Fix $k\in\NN$, and let $\{z^{\alpha(j)}\}_{j=1}^{M_k}$ be the enumeration of the monomials in $\Poly(kC)$ according to grevlex, i.e., $z^{\alpha(j)}\prec z^{\alpha(k)}$ if $j<k$.

Let us introduce the following notation: for $j=1,\ldots,M_k$, set  
$$
W_j(\zeta_1,\ldots,\zeta_{j}) := \left| \det 
\begin{bmatrix}
w^k(\zeta_1)z^{\alpha(1)}(\zeta_1) & \cdots & w^k(\zeta_{j})z^{\alpha(1)}(\zeta_{j}) \\
\vdots & \ddots & \vdots \\
w^k(\zeta_1)z^{\alpha(j)}(\zeta_1) & \cdots & w^k(\zeta_{j})z^{\alpha(j)}(\zeta_{j})
\end{bmatrix} \right|.
$$
Observe that for any collection of points $\{\zeta_1,\ldots,\zeta_{M_k}\}\subseteq K$, 
$$
V^w_{k}(K,M_k) \geq W_{M_k}(\zeta_1,\ldots,\zeta_{M_k}).
$$
We now derive an inequality involving $W_{M_k-1}$. Fix   $\{\zeta_1,\ldots,\zeta_{M_k-1}\}\subseteq K$.  Then, by considering row operations as before,   $W_{M_k}(\zeta_1,\ldots,\zeta_{M_k-1},\eta)$ is equal to 
$$
\left| \det
\begin{bmatrix}
w^k(\zeta_1)z^{\alpha(1)}(\zeta_1) & \cdots & w^k(\zeta_{M_k-1})z^{\alpha(1)}(\zeta_{M_k-1}) & w^k(\eta)z^{\alpha(1)}(\eta) \\
\vdots & \ddots & \vdots & \vdots \\
w^k(\zeta_{1})z^{\alpha(M_k-1)}(\zeta_1) & \cdots & w^k(\zeta_{M_k-1})z^{\alpha(M_k-1)}(\zeta_{M_k-1}) &  w^k(\eta)z^{\alpha(M_k-1)}(\eta) \\
w^k(\zeta_1)\bp_{M_k}(\zeta_1) & \cdots & w^k(\zeta_{M_k-1})\bp_{M_k}(\zeta_{M_k-1}) & w^k(\eta)\bp_{M_k}(\eta)
\end{bmatrix} \right| 
$$
for any $\bp_{M_k}\in\calM_k^{\prec}(\alpha(M_k))$ and $\eta\in K$.  We choose the particular multiples of previous rows that give a polynomial satisfying  $\bp_{M_k}(\zeta_1)=\cdots=\bp_{M_k}(\zeta_{M_k-1})=0$.  Then $\bp_{M_k}\in\calM_k^{\prec}(\alpha(M_k))$ and 
\begin{align*}
&  V_{k}^w(K,M_k)  \\
& \quad \geq 
\left|\det
\begin{bmatrix}
w^k(\zeta_{1})z^{\alpha(1)}(\zeta_1) & \cdots & w^k(\zeta_{M_k-1})z^{\alpha(1)}(\zeta_{M_k-1}) &  w^k(\eta)z^{\alpha(1)}(\eta) \\
\vdots & \ddots & \vdots & \vdots \\
w^k(\zeta_{1})z^{\alpha(M_k)}(\zeta_1) & \cdots & w^k(\zeta_{M_k-1})z^{\alpha(M_k)}(\zeta_{M_k-1}) &  w^k(\eta)z^{\alpha(M_k)}(\eta) \\
0 & \cdots & 0 & w^k(\eta)\bp_{M_k}(\eta)
\end{bmatrix} \right| \\
& \quad =  |w^k(\eta)\bp_{M_k}(\eta)|\cdot|W_1(\zeta_1,\ldots,\zeta_{M_k-1})|.
\end{align*}
Choose $\eta\in K$ that attains the sup norm, so that $$|w^k(\eta)\bp_{M_k}(\eta)|=\|w^k\bp_{M_k}\|_K\geq T^{\prec}_k(\nu_k,\alpha(M_k))^k;$$ then
\begin{equation}\label{eqn:W}
V_{k}^w(K,M_k)\geq  |W_{M_k}(\zeta_1,\ldots,\zeta_{M_k-1},\eta)| \geq   |W_{M_k-1}(\zeta_1,\ldots,\zeta_{M_k-1})|T^{\prec}_k(\nu_k,\alpha(M_k))^k.
\end{equation}

Observe that $\zeta_1,\ldots,\zeta_{M_k-1}$ are arbitrary points of $K$ in the above equation.  
Now consider fixing $\zeta_1,\ldots,\zeta_{M_k-2}$ and carry out a similar argument: construct a polynomial $\bp_{M_k-1}\in\calM_k^{\prec}(\alpha(M_k-1))$ with $\bp_{M_k-1}(\zeta_1)=\cdots=\bp_{M_k-1}(\zeta_{M_k-2})=0$, substitute it for $z^{\alpha(M_k-1)}$ in the second to last row, then choose the last point $\eta\in K$ such that $|w^k(\eta)\bp_{M_k-1}(\eta)|=\|w^k\bp_{M_k-1}\|_K$. This gives the inequality 
\begin{equation}\label{eqn:W1}
W_{M_k-1}(\zeta_1,\ldots,\zeta_{M_k-2},\eta)   \geq W_{M_k-2}(\zeta_1,\ldots,\zeta_{M_k-2}) T_k^{\prec}(\nu_k,\alpha(M_k-1))^k.
\end{equation}
Use (\ref{eqn:W}) to estimate the left-hand side of (\ref{eqn:W1}), observing that the upper bound on the left-hand side of (\ref{eqn:W}) is valid for an arbitrary collection of $M_{k-1}$ points of $K$.  Hence 
$$V_{k}^w(K,M_k) \geq W_{M_k-2}(\zeta_1,\ldots,\zeta_{M_k-2}) T_k^{\prec}(\nu_k,\alpha(M_k-1))^k T_k^{\prec}(\nu_k,\alpha(M_k))^k .$$
Now it is easy to see that the argument can be iterated to obtain the estimate 
\begin{equation}\label{eqn:W2}
V_{k}^w(K,M_k) \geq W_s(\zeta_1,\ldots,\zeta_{s}) \Bigl(\prod_{j=s+1}^{M_k} T_k^{\prec}(\nu_k,\alpha(j))^k \Bigr) 
\end{equation}
 for successively lower values of  $s$.  The first statement of the lemma is proved when $s=1$.

To get the inequality with each $T_k^{\prec}(\nu_k,\alpha(j))$ replaced by $T_k^C(\nu_k,\alpha(j))$, repeat the proof word for word, but enumerate the monomials using $\prec_C$.
\end{proof}

Using these estimates, we can now prove transfinite diameter formulas.  As mentioned in (\ref{rem:lk}), the $L_k$-th root is chosen so that the quantities $\delta_{C,k}^w$, $\delta_C^w$ scale like a length.  However it is more natural to take the $kM_k$-th root in the calculations that follow.       

\begin{theorem}\label{thm:Ok}
The limit $\displaystyle D^w(K) := \lim_{k\to\infty} (V_{k}^w(K,M_k))^{1/(kM_k)}$, exists, and
\begin{equation} \label{eqn:transfd}
D^w(K)
= \exp\left(\frac{1}{\vol_N(C)}\int_{\intt(C)} \log T_{\prec}^w(K,\theta)d\theta \right).
\end{equation}
 The transfinite diameter is then given by 
$$
\delta^w_C(K) = D^w(K)^{1/A_N}.
$$
\end{theorem}


\begin{proof}
We have by Lemmas \ref{lem:<} and \ref{lem:>}, 
$$ 
 \prod_{j=1}^{M_k} T_k^{\prec}(\nu_k,\alpha(j))^{k}    \leq	  V_{k}^w(K,M_k)   \leq  M_k!\prod_{j=1}^{M_k} T^{\prec}_k(\nu_k,\alpha(j))^k .
$$
Let us take $kM_k$-th roots in this inequality and let $k\to\infty$.  Since   
$(M_k!)^{1/(kM_k)}\to 1$ as  $k\to\infty$,  
$$ 
D^w(K) = \lim_{k\to\infty} \left(V_{k}^w(K,M_k)\right)^{1/(kM_k)} =\lim_{k\to\infty} \left(\prod_{j=1}^{M_k} T^{\prec}_k(\nu_k,\alpha(j))^k\right)^{1/(kM_k)}.  $$
We need to show that the limit on the right-hand side of the above converges to the right-hand side of (\ref{eqn:transfd}).  To see this, observe that the limit may be rewritten as
\begin{equation}\label{eqn:pf312}
\lim_{k\to\infty} \exp\biggl(   
\frac{1}{M_k}\sum_{j=1}^{M_k} \log T^{\prec}_k(\nu_k,\alpha(j)) 
\biggr).
\end{equation}
We look at the limit of the expression inside the parentheses.  If we fix a compact convex body $Q\subset\intt(C)$, then
$$
\lim_{k\to\infty} \frac{1}{M_k}\sum_{\alpha(j)\in kQ}^{M_k} \log T^{\prec}_k(\nu_k,\alpha(j)) \ = \ \lim_{k\to\infty} \frac{1}{M_k}\sum_{\alpha(j)\in kQ} \log T^w_{\prec}(K,\theta_j)
$$ 
where $\theta_j=\alpha(j)/k$, since for fixed $k$ the difference between the quantities on each side is bounded above by 
$$
\sup\left\{ |\log|T^{\prec}_k(\nu_k,\alpha(j)) - \log|T^w_{\prec}(K,\theta_j)|\colon j\in\{1,\ldots,N_k\}  \right\},
$$
and this goes to zero by Corollary \ref{cor:dircheby}   (to Lemma \ref{lem:riemann}).

Since $\{\theta_j\colon j=1,\ldots,M_k\} = \tfrac{1}{k}\ZZ^N\cap C$, the discrete measure $\frac{1}{M_k}\sum_{j=1}^{M_k}\delta_{\theta_j}$, supported on a uniform grid, clearly converges weak-$^*$ on $\intt(C)$ to the uniform measure  $\frac{1}{\vol_N(C)}\,d\theta$ as $k\to\infty$.  Hence 
$$
\lim_{k\to\infty} \frac{1}{M_k}\sum_{\theta_j\in Q} \log T^w_{\prec}(K,\theta_j) \  = \   \frac{1}{\vol_N(C)}\int_Q\log T^w_{\prec}(K,\theta)\, d\theta.
$$
Since $Q$ was arbitrary, one can consider a sequence of compact convex sets increasing to $\intt(C)$ and obtain
$$
\lim_{k\to\infty} \frac{1}{M_k}\sum_{j=1}^{M_k} \log T^w_{\prec}(K,\theta_j) \  = \   \frac{1}{\vol_N(C)}\int_{\intt(C)}\log T^w_{\prec}(K,\theta)\, d\theta,
$$
by a dominated convergence argument.

The formula for transfinite diameter follows by applying the third limit in Lemma \ref{lem:limits}.
\end{proof}

We have a corresponding integral formula in terms of $T^w_C(K,\theta)$.

\begin{theorem}\label{thm:relate}
We have 
$$
D^w(K) =  \exp\left(\frac{1}{\vol_N(C)}\int_{\intt(C)} \log T_{C}^w(K,\theta)d\theta \right).
$$ 
\end{theorem}

\begin{proof}  By Lemmas \ref{lem:<} and \ref{lem:>} we have 
\begin{equation}\label{eqn:relate}
\tfrac{1}{M_k!}V_k^w(K,M_k) \leq  \prod_{\alpha\in kC} T_k^C(\nu_k,\alpha)^k \leq V_k^w(K,M_k) \quad\hbox{for any }k\in\NN,
\end{equation}
and taking logs we obtain
\begin{equation}\label{eqn:logTk}
-\frac{\log(M_k!)}{kM_k} + \log V_k^w(K,M_k)^{\frac{1}{kM_k}} \leq 
\frac{1}{M_k}\sum_{\alpha\in kC} \log T_k^C(\nu_k,\alpha)  \leq \log V_k^w(K,M_k)^{\frac{1}{kM_k}}.
\end{equation}
We will apply measure theory to this estimate.  But first we will take care of some technicalities before returning to the proof. \end{proof}

Introduce the following notation.  Given $\alpha\in\ZZ^N$, let
$$
E_{\alpha,k}:=\{\theta\in\RR^N\colon |\theta-\alpha/k|<|\theta-\beta/k| \hbox{ for any } \beta\in\ZZ^N\},
$$
\begin{multicols}{2}
\noindent and let $\bar E_{\alpha,k}$ be its closure.  By elementary geometry, $E_{\alpha,k}$ is an open $N$-dimensional cube centered at $\alpha/k$, with side length $1/k$ and faces parallel to the coordinate hyperplanes.  The collection $\{E_{\alpha,k}\}_{\alpha\in\ZZ^N}$ forms a grid in $\RR^N$.

 \begin{center} \includegraphics[height=3cm]{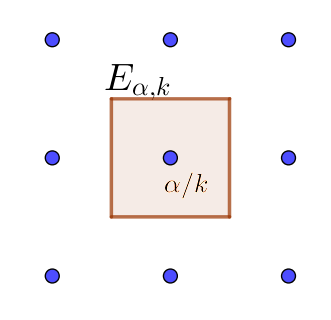} \end{center}
\end{multicols} 

Define the following subsets of $\ZZ^N$:
\begin{align*}
\calN_k&:= \{\alpha\in\ZZ^N\colon \alpha\in kC\}, \\
\bar\calN_k &:= \{\alpha\in\ZZ^N\colon \bar E_{\alpha,k}\cap C\neq\emptyset\}, \\
\partial\calN_k &:= \{\alpha\in\ZZ^N\colon \bar E_{\alpha,k}\cap\partial C\neq\emptyset\}, \\
\calN_k^{\circ} &:= \bar\calN_k\setminus\partial\calN_k.
\end{align*}
By the notation defined earlier, $M_k$ is the size of $\calN_k$. We will denote by $\bar M_k$, $B_k$ and $M_k^{\circ}$ the sizes of the other three sets, respectively.  (So $M_k^{\circ}=\bar M_k-B_k$.)

\begin{lemma}\label{lem:Nk}
We have $$\calN_k^{\circ}\subseteq\{\alpha\in\ZZ^N\colon E_{\alpha,k}\subseteq\intt(C)\} \  \hbox{ and } \   \calN_k^{\circ}\subseteq \calN_k\subseteq \bar\calN_k.$$

  Also, $\frac{B_k}{M_k}=O(\frac{1}{k})$, so that $\frac{B_k}{M_k},\frac{B_k}{M_k^{\circ}}\to 0$ and   $\frac{\bar M_k}{M_k}, \frac{M_k^{\circ}}{ M_k}\to 1$ as $k\to\infty$.  
\end{lemma}

\begin{proof}
If $\alpha\in kC$ then $E_{\alpha,k}\cap C\neq\emptyset$, so $\calN_k\subseteq\bar\calN_k$ follows.  If $\alpha\in\calN_k^{\circ}$, then $\bar E_{\alpha,k}\cap\partial C=\emptyset$.  It follows, since $C$ and $\bar E_{\alpha,k}$ are connected, that $E_{\alpha,k}\subseteq\intt(C)$.  This also implies that $\alpha\in\intt(kC)$, so $\calN_k^{\circ}\subseteq\calN_k$.  The first statement is proved.

By construction, $\{\bar E_{\alpha}\}_{\alpha\in\bar\calN_k}$ and 
$\{\bar E_{\alpha}\}_{\alpha\in\partial\calN_k}$ are precisely those cubes  that   
 cover $C$ and $\partial C$ respectively.  We claim that $M_k=O(k^N)$ and $B_k=O(k^{N-1})$, which is the same as saying that the box-counting dimensions of $C,\partial C$ are $N,N-1$ respectively.  This is true because $C$ is a convex set in $\RR^N$ with nonempty interior.\footnote{In particular, $\partial C$ is locally the graph of a Lipschitz map, so the box-counting dimension is the same as that of an ($N-1$)-dimensional hyperplane (i.e., $N-1$).}    Hence $M_k=O(k^N)$ and $B_k=O(k^{N-1})$, and the second statement  follows immediately.
\end{proof}



Define the function $\ell_k:C\to[0,\infty)$ by 
\begin{equation}\label{eqn:lk}
\ell_k(\theta) := \sum_{\alpha\in\calN_k} \log T_k^C(\nu_k,\alpha)\chi_{E_{\alpha,k}}(\theta)
\end{equation}
where $\chi_{E_{\alpha,k}}$ denotes the characteristic function of $E_{\alpha,k}$. So $\ell_k(\theta) = T_k^C(\nu_k,\alpha)$ for $\theta\in E_{\alpha,k}$; in particular, when $|\theta-\alpha/k|\leq \frac{N}{2k}$.  Together with Lemma \ref{lem:32}, we have
\begin{equation}\label{eqn:412} 
\ell_k(\theta)\to \log T^w_C(K,\theta)  \hbox{ pointwise, as } k\to\infty.\end{equation}
Now 
$$
\int_{\intt(C)} \ell_k(\theta) d\theta = \sum_{\alpha\in\calN_k} \log T_k^C(\nu_k,\alpha)\vol_N(E_{\alpha,k}\cap C),
$$
 so by the first statement of Lemma \ref{lem:Nk}, we have the bounds
\begin{equation}\label{eqn:bound1}
\sum_{\alpha\in\calN_k^{\circ}} \log T_k^C(\nu_k,\alpha)\vol_N(E_{\alpha,k}) \leq \int_{\intt(C)} \ell_k(\theta) d\theta
 \leq \sum_{\alpha\in\bar\calN_k} \log T_k^C(\nu_k,\alpha)\vol_N(E_{\alpha,k}).
\end{equation}
Similarly,
\begin{equation}\label{eqn:bound2}
\frac{M_k^{\circ}}{k^N} = \sum_{\alpha\in\calN_k^{\circ}}\vol_N(E_{\alpha,k}) \leq \vol_N(C) \leq \sum_{\alpha\in\bar\calN_k}\vol_N(E_{\alpha,k})  =  \frac{\bar M_k}{k^N} .
\end{equation}

\begin{lemma}\label{lem:lk}
We have 
$$
\left| \frac{1}{\vol_N(C)}\int_{\intt(C)}\ell_k(\theta) d\theta - \frac{1}{M_k}\sum_{\alpha\in\calN_k}\log T_k^C(\nu_k,\alpha) \right| \to 0 \ \hbox{ as } k\to\infty.
$$
\end{lemma}

\begin{proof}
We will prove the lemma by showing that   
\begin{equation}   \label{eqn:upper1}
\limsup_{k\to\infty}\left(\frac{1}{\vol_N(C)}\int_{\intt(C)} \ell_k(\theta) d\theta - \frac{1}{M_k}\sum_{\alpha\in\calN_k}\log T_k^C(\nu_k,\alpha)\right) \leq 0  \end{equation}
and
\begin{equation}
\liminf_{k\to\infty}\left(\frac{1}{\vol_N(C)}\int_{\intt(C)} \ell_k(\theta) d\theta - \frac{1}{M_k}\sum_{\alpha\in\calN_k}\log T_k^C(\nu_k,\alpha)\right) \geq 0.  \label{eqn:lower1}
\end{equation}
By (\ref{eqn:bound1}), (\ref{eqn:bound2}) and the fact that $\vol_N(E_{\alpha,k})=\left(\tfrac{1}{k}\right)^N$ for all $\alpha$, 
\begin{align*}
\frac{1}{\vol_N(C)}\int_{\intt(C)} l_k(\theta) d\theta  
&\leq \frac{k^N}{M_k^{\circ}}\sum_{\alpha\in\bar\calN_k}\log T_k^C(\nu_k,\alpha)\left(\tfrac{1}{k}\right)^N \\
&= \frac{1}{M_k^{\circ}}\sum_{\alpha\in\bar\calN_k}\log T_k^C(\nu_k,\alpha).
\end{align*}
Hence
\begin{align*}
&\frac{1}{\vol_N(C)}\int_{\intt(C)} \ell_k(\theta) d\theta - \frac{1}{M_k}\sum_{\alpha\in\calN_k}\log T_k^C(\nu_k,\alpha) \\
&\qquad \leq \frac{1}{M_k^{\circ}}\sum_{\alpha\in\bar\calN_k}\log T_k^C(\nu_k,\alpha) - \frac{1}{M_k}\sum_{\alpha\in\calN_k^{\circ}}\log T_k^C(\nu_k,\alpha) \\
&\qquad = \left(\frac{1}{M_k^{\circ}}-\frac{1}{M_k}\right)\sum_{\alpha\in\calN_k^{\circ}}\log T_k^C(\nu_k,\alpha) + \frac{1}{M_k^{\circ}}\sum_{\alpha\in\partial\calN_k}\log T_k^C(\nu_k,\alpha) \\
&\qquad \leq \left(\frac{1}{M_k^{\circ}}-\frac{1}{M_k}\right)M_k^{\circ}\log R + \frac{1}{M_k^{\circ}}B_k\log R\\ 
&\qquad = \left( 1 - \tfrac{M_k^{\circ}}{M_k} + \tfrac{B_k}{M_k^{\circ}}  \right)\log R \longrightarrow 0 \quad\hbox{as } k\to 0
\end{align*}
by Lemma \ref{lem:Nk}.  Hence we obtain (\ref{eqn:upper1}).  Next, using (\ref{eqn:bound1}), (\ref{eqn:bound2}) to estimate from below, we obtain by a similar calculation that
\begin{align*}
\frac{1}{\vol_N(C)}\int_{\intt(C)} \ell_k(\theta) d\theta - \frac{1}{M_k}\sum_{\alpha\in\calN_k}\log T_k^C(\nu_k,\alpha) 
&\geq \left(\tfrac{M_k^{\circ}}{\bar M_k} - \tfrac{M_k^{\circ}}{M_k} - \tfrac{B_k}{M_k}\right)  \log R \\
 &\longrightarrow 0 \quad  \hbox{as }k\to 0, 
\end{align*}
which proves (\ref{eqn:lower1}).
\end{proof}

\begin{proof}[End of the proof of Theorem \ref{thm:relate}]
By the estimate (\ref{eqn:logTk}) and Theorem \ref{thm:Ok},
\begin{equation*}
\lim_{k\to\infty}\left(\frac{1}{M_k}\sum_{\alpha\in\calN_k}\log T^C_k(\nu_k,\alpha) \right)= \log D^w(K). 
\end{equation*}
Applying Lemma \ref{lem:lk}, 
$$
\lim_{k\to\infty} \left( \frac{1}{\vol_N(C)}\int_{\intt(C)}\ell_k(\theta)d\theta \right) = \log D^w(K).  
$$
Finally, using (\ref{eqn:412}) and the dominated convergence theorem, the theorem follows.
\end{proof}

\section{$C$-Leja points}

To define $C$-Leja points, (called Leja points from now on, for convenience)  we use the modified grevlex ordering $\prec_C$; enumerate the monomials $\{z^{\alpha(j)}\}_{j=1}^{\infty}$ in increasing order according to $\prec_C$.  We will also write 
$$VDM(\zeta_1,\ldots,\zeta_j):=VDM_k^w(\zeta_1,\ldots,\zeta_j)$$ where $k= k(j):=\deg_C(z^{\alpha(j)})$ and $w$ is a fixed weight function on $K$.  

A sequence of Leja points for $K$ is a sequence 
$$\calL:=\{\zeta_j\}_{j=1}^{\infty}=\bigcup_{s\in\NN}\calL_s,$$   where $\calL_s=\{\zeta_j\}_{j=1}^s$ is constructed inductively:   first set   $\calL_1:=\{\zeta_1\}$ for some $\zeta_1\in K$ for which $w(\zeta_1)\neq 0$; then given $\calL_s$, form $\calL_{s+1}:=\calL_s\cup\{\zeta_{s+1}\}$ by picking $\zeta_{s+1}\in K$ that satisfies 
$$
|VDM(\zeta_1,\ldots,\zeta_s,\zeta_{s+1})| = \sup_{\eta\in K} |VDM(\zeta_1,\ldots,\zeta_s,\eta)|.
$$
The Leja points of order $k$ are those of the set  $\calL_{M_k}$.

Assume condition $(\dag)$ of Section \ref{sec:3} holds for $C$.  We now show that in the unweighted case ($w\equiv 1$), Leja points give the transfinite diameter $\delta_C(K)$.  The statement for weighted Leja points is an open conjecture, even in the classical case $C=\Sigma$.

\begin{theorem}\label{thm:leja}
Let $\{\zeta_j\}_{j=1}^{\infty}$ be a sequence of (unweighted) Leja points for $K$, and let $L_s:=VDM(\zeta_1,\ldots,\zeta_s)$.  Then
\begin{equation}\label{eqn:thm1}
\lim_{k\to\infty} (L_{M_k})^{1/L_k} = \delta_C(K).
\end{equation} 
\end{theorem}

Before beginning the proof, note that by definition, $L_{M_k}\leq V_{M_k}$, so $\delta_C(K)$ is an upper bound for the lim sup of the left-hand side as $k\to\infty$.  To get a lower bound for the lim inf we will bound $L_{M_k}$ from below.  The argument is similar to that of Proposition 3.7 in  \cite{bloomboschristensenlev:polynomial}.     First we need a lemma comparing Chebyshev constants.  

\begin{lemma}\label{lem:12}
Let $\alpha\in\ZZ^N$ and $j,k\in\NN$ with $\alpha\in jC\subset kC$.  Then 
$$
T_k^C(\nu_k,\alpha)^k = T_j^C(\nu_j,\alpha)^j.
$$
\end{lemma}
\begin{proof}
By the definition of $\prec_{C}$, we have $\calM_j^C(\alpha)=\calM_k^C(\alpha)$, as all lower terms for the latter class are in $jC$.  Since the classes are the same, the Chebyshev constants are the same: if $p$ is a polynomial in the class that maximizes $\|p\|_K$, then 
$T_k^C(\nu_k,\alpha)^k  = \|p\|_K = T_j^C(\nu_j,\alpha)^j.$
\end{proof}

\begin{remark}\rm
The above proof only works in the unweighted setting.   If $w\not\equiv 1$ and $j\neq k$ then $\|w^kp\|_K$ and $\|w^jp\|_K$ are not necessarily equal.
\end{remark}


\begin{lemma} Let $s\in\NN$.  Then $L_s \geq T^C_k(\alpha(s))^kL_{s-1}.$
\end{lemma}

\begin{proof}
By the definition of $L_s$, the point $\zeta_{s}\in K$ maximizes $VDM(\zeta_1,\ldots,\zeta_{s-1},\eta)$ over all $\eta\in K$, which by the proof of Lemma \ref{lem:>}, maximizes $$|VDM(\zeta_1,\ldots,\zeta_{s-1})| \cdot|p(\eta)|$$ over all $\eta\in K$ for some polynomial $p\in\calM^C_k(\alpha(s))$.  Hence
\begin{align*}
L_s &= |VDM(\zeta_1,\ldots,\zeta_s)| \\ 
 &=  |p(\zeta_s)|\cdot|VDM(\zeta_1,\ldots,\zeta_{s-1})| \\
&=  \|p\|_K |VDM(\zeta_1,\ldots,\zeta_{s-1})|  
=   \|p\|_K L_{s-1}   \geq  T^C_{k}(\alpha(s))^kL_{s-1}.
\end{align*}
\end{proof}

\begin{proof}[Proof of Theorem \ref{thm:leja}]
Let $k\in\NN$.  By repeatedly applying the previous lemma for all indices in $kC\setminus(k-1)C$, we obtain
$$
L_{M_k}\geq L_{M_{k-1}}\prod_{s=M_{k-1}+1}^{M_k} T^C_k(\alpha(s))^k .
$$
Iterating the above estimate another $k-1$ times, we obtain  
\begin{equation*}
L_{M_{k}} \geq \prod_{j=0}^k \left(     \prod_{s=M_{j-1}+1}^{M_{j}} T^C_j(\alpha(s))^j  \right) = 
 \prod_{s=1}^{M_{k}} T^C_{k}(\alpha(s))^{k}
\end{equation*}
where the last equality follows from Lemma \ref{lem:12}.  Putting the right-hand side into the estimate (\ref{eqn:relate}), we obtain 
$$
L_{M_k} \geq \tfrac{1}{M_k!}V_k(K,M_k).
$$
Taking $L_k$-th roots and letting $k\to\infty$, the result follows by Theorem \ref{thm:Ok}.
\end{proof}

\begin{remark}\rm
Let $\mu_k:=\tfrac{1}{M_k}\sum_{j=1}^{M_k} \delta_{\zeta_j}$  be the discrete probability measure at the Leja points of order $k$.  Putting together Theorem \ref{thm:leja} and Corollary 6.5 from \cite{bayraktarbloomlev:pluripotential}, we have weak-$^*$ convergence to the equilibrium measure:
$$
\mu_k \to (dd^cV_{C,K})^N \hbox{ as } k\to\infty.
$$
\end{remark}

\medskip

\subsection*{Acknowledgement}  
I would like to thank Norm Levenberg for pointing out an error in an earlier draft of the paper.


\begin{thebibliography}{10}

\bibitem{bayraktar:zero}
Turgay Bayraktar.
\newblock Zero distribution of random sparse polynomials.
\newblock {\em Michigan Math. J.}, 66:389--419, 2017.

\bibitem{bayraktarbloomlev:pluripotential}
Turgay Bayraktar, Thomas Bloom, and Norm Levenberg.
\newblock Pluripotential theory and convex bodies.
\newblock {\em Mat. Sb.}, 209:352--384, 2018.

\bibitem{bermanboucksom:growth}
Robert Berman and S\'ebastien Boucksom.
\newblock Growth of balls of holomorphic sections and energy at equilibrium.
\newblock {\em Invent. Math.}, 181:337--394, 2010.

\bibitem{bermanboucksomnystrom:fekete}
Robert Berman, S\'ebastien Boucksom, and David~Witt Nystr\"om.
\newblock Fekete points and convergence towards equilibrium measures on complex
  manifolds.
\newblock {\em Acta. Math.}, 207(1):1--27, 2011.

\bibitem{bloomboschristensenlev:polynomial}
Thomas Bloom, Len Bos, C.~Christensen, and Norman Levenberg.
\newblock Polynomial interpolation of holomorphic functions in $\mathbb{C}$ and
  $\mathbb{C}^n$.
\newblock {\em Rocky Mountain J. Math}, 22(2):441--470, 1992.

\bibitem{bloomlev:weighted}
Thomas Bloom and Norman Levenberg.
\newblock Weighted pluripotential theory in $\mathbb{C}^n$.
\newblock {\em Amer. J. Math.}, 125(1):57--103, 2003.

\bibitem{boslev:bernstein}
Len Bos and Norm Levenberg.
\newblock Bernstein-{W}alsh theory associated to convex bodies and applications
  to multivariate approximation theory.
\newblock {\em Comput. Methods Funct. Theory}, 18:361--388, 2018.

\bibitem{mau:newton}
Sione Ma`u.
\newblock {N}ewton-{O}kounkov bodies and transfinite diameter.
\newblock {\em Dolomites Res. Notes Approx.}, 10(Special {I}ssue):138--160,
  2017.

\bibitem{trefethen:multivariate}
Lloyd~N. Trefethen.
\newblock Multivariate polynomial approximation in the hypercube.
\newblock {\em Proc. Amer. Math. Soc.}, 145(11):4837--4844, 2017.

\bibitem{nystrom:transforming}
David {W}itt {N}ystr\"{o}m.
\newblock Transforming metrics on a line bundle to the {O}kounkov body.
\newblock {\em Ann. Sci. \'{E}c. Norm. Sup\'{e}r.}, 47(4):1111--1161, 2014.

\bibitem{Zaharjuta:transfinite}
V.P. Zaharjuta.
\newblock Transfinite diameter, {C}hebyshev constants, and capacity for
  compacta in $\mathbb{C}^n$.
\newblock {\em Math. USSR Sbornik}, 25(3):350--364, 1975.

\end{thebibliography}

\end{document}